\documentclass{amsart}
\usepackage{graphicx}
\usepackage{amssymb}
\newtheorem{thm}{Theorem}[section]

\newtheorem{alg}[thm]{Algorithm}
\newtheorem{exmp}[thm]{Example}

\theoremstyle{definition}
\newtheorem{defn}[thm]{Definition}
\theoremstyle{remark}

\numberwithin{equation}{section}

\begin{document}

\title[Subalgebra Analogue to Standard basis for Ideal ]{Subalgebra Analogue to Standard basis for Ideal }%
\author{Junaid Alam Khan$^*$}%
\address{$*$Abdus Salam School of Mathematical Sciences, GCU, Lahore Pakistan}%
\email{junaidalamkhan$@$gmail.com}%

\thanks{This research was partially supported by Higher Education Commission, Pakistan  }
\subjclass[2000]{Primary 13P10, 13J10;} 
\keywords{Standard basis, Sagbi basis, Local ordering}%

\begin{abstract}
The theory of ``subalgebra basis" analogous to standard basis (the generalization of Gr\"{o}bner bases to monomial ordering which are not necessarily well ordering \cite{GP1}.) for ideals in polynomial rings over a field is developed. We call these bases ``SASBI Basis" for ``Subalgebra Analogue to Standard Basis for Ideals". The case of global orderings, here they are called ``SAGBI Basis" for ``Subalgebra Analogue to Gr\"{o}bner Basis for Ideals", is treated in \cite{RS1}. Sasbi bases may be infinite. In this paper we consider subalgebras admitting a finite Sasbi basis and give algorithms to compute them. The algorithms have been implemented as a library for the computer algebra system SINGULAR \cite{GPS1}.
\end{abstract}
\maketitle
\section{Introduction and Preliminaries}
Let $K$ be a field and  $K[x_1,\ldots,x_n]$ the polynomial the ring over the field $K$ in $n$ variables and $K[[x_1,\ldots,x_n]]$ the formal power series ring. Let $G\subset \mathcal{M} \backslash \{0\}$ where $\mathcal{M}$ is the maximal ideal of $ K[[x_1,\ldots,x_n]]$. We define  $$K[[G]]=\{Q(g_1,\ldots,g_s)\,|\,Q\in K[[y_1,\ldots,y_s]]\,\,\hbox{and}\,\,g_1,\ldots,g_s\in G\,\,\hbox{for some}\,s \}.$$ In analogy to the theory of Gr\"{o}bner bases for ideals in $K[x_1,\ldots,x_n]$ resp. standard bases for ideals in $K[[x_1,\ldots,x_n]]$ there exist a theory of Gr\"{o}bner bases for subalgebras of type $K[G]$ called Sagbi basis (c.f \cite{RS1}) and of standard bases of subalgebras of type $K[[G]]$ developed in \cite{HH1}. Let $f_{1},\dots,f_{m}\in K[x_1,\ldots,x_n]$ and consider the ideal $I=\langle f_1,\ldots,f_m\rangle K[[x_1,\ldots,x_n]] $. Suppose we want  to compute a standard basis of  $I$. There are at least three possibilities. Using Buchberger's algorithm for well orderings we can compute it upto a given degree. There exist a theory of standard bases in $K[x_1,\ldots,x_n]_{\langle x_1,\ldots,x_n\rangle}$ induced by a local degree ordering (see \cite{GP1}) and we can  compute a standard basis $g_1,\ldots,g_s\in K[x_1,\ldots,x_n]$ of $\langle f_1,\ldots,f_m\rangle K[x_1,\ldots,x_n]_{\langle x_1,\ldots,x_n\rangle}$ using Mora's tangent cone algorithm (c.f \cite{GP1}, \cite{MPT1}). It can also be computed via homogenization (Lazard's algorithm c.f \cite{L1}). In this paper we will developed a subalgebra analogue of standard bases of ideals in $K[x_1,\ldots,x_n]_{\langle x_1,\ldots,x_n\rangle}$. We will introduce subalgebra bases in a suitable localization of $K[G]$, which we call it Sasbi bases. They can be computed upto a certain degree similar to the standard basis case. They can be computed via a Sagbi bases if  the homogenized algebra has a finite Sagbi basis. In this case they can also be computed directly using a generalization of the tangent cone algorithm. We will show that  Sasbi bases are subalgebra standard bases of $K[[G]]$. Using homogenization we will get a finiteness  condition. If $K[G^{h}]$ has a finite Sagbi basis then $K[[G]]$ has a finite subalgebra standard basis. The aim of this paper is to give an algorithm to compute these bases provided $G\subset K[[x_1,\ldots,x_n]]$ is a finite set and the subalgebra standard basis for $K[[G]]$ is finite\footnote{They may be infinite as $K[[x_1,\ldots,x_n]]$ doesn't satisfy the ascending chain condition with respect to subalgebras. }.
We use the notations from \cite{GP1} and repeat them for the convenience of reader.
\begin{defn}
A monomial ordering is a total ordering $>$ on the set of monomials $Mon_n=\{x^\alpha \,|\, \alpha \in \mathbb{N}^n\}$ in n variables satisfying $$x^\alpha > x^\beta \Longrightarrow x^\gamma x^\alpha > x^\gamma x^\beta$$ for all $\alpha, \beta, \gamma  \in \mathbb{N}^n$. We also say $>$ is a monomial ordering on $K[x_1,\ldots, x_n]$  meaning that $>$ is a monomial ordering on $Mon_n$.
\end{defn}

\begin{defn}
Let $>$ be a fixed monomial ordering. Write $f\in K[x_1,\ldots,x_n]$, $f\neq 0$, in a unique way as a sum of non-zero terms $$f=a_\alpha x^\alpha +a_\beta x^\beta +\ldots +a_\gamma x^\gamma , \, \,\,\,\,\,\,\,\,\, x^\alpha > x^\beta > \ldots  > x^\gamma ,$$ and $a_\alpha ,a_\beta ,\ldots a_\gamma \in K$. We define:
\begin{itemize}
\item[1.]$LM(f):=x^\alpha$, the leading  monomial of $f$,
\item[2.]$LE(f):=\alpha$, the leading  exponent of $f$,
\item[3.]$LT(f):=a_\alpha x^\alpha$, the leading  term of  $f$,
\item[4.]$LC(f):=a_\alpha$, the leading  monomial of $f$,
\item[5.]$tail(f):=f-LT(f)$.
\item[6.]${ecart}(f):={deg}(f)-{deg(LM}(f)).$
\item[7.]$support(f):=\{x^{\alpha},x^{\beta},\ldots,x^{\gamma}\}$, the set of all monomials of $f$ with non-zero coefficent.
\item[8.]$ord(f)=deg(LM(f)).$
\end{itemize}
\end{defn}

\begin{defn}

Let $>$ be a monomial ordering on $Mon_n$.
\\1. $>$ is called global ordering if $x^\alpha > 1$ for all $\alpha \neq (0, \ldots ,0)$.
\\2. $>$ is called local ordering if $x^\alpha < 1$ for all $\alpha \neq (0, \ldots ,0)$.
\\3. $>$ is called local degree ordering, if $>$ is a local ordering and $$x^{\alpha}>x^{\beta}\,\Rightarrow deg(x^{\alpha})\leq deg(x^{\beta})$$
\end{defn}
\begin{defn}
Let $M\in GL(n,\mathbb{R})$. We can use $M$  to obtain a monomial ordering by setting
   $$x^{\alpha}>_{M}x^{\beta}\,:\Longleftrightarrow\,M{\alpha}>M{\beta},$$
where $>$ on the right-hand side is the lexicographical ordering on $\mathbb{R}^{n}$.

\end{defn}

\begin{thm}
(c.f. \cite{GP1}, page 18 )Any monomial ordering can be defined as $>_{M}$ by a matrix $M\in GL(n,\mathbb{R})$.

\end{thm}
\begin{defn}
For any monomial ordering $>$ on $Mon(x_1, \ldots ,x_n)$, we define a multiplicatively closed set.    $$S_>:=\{u\in K[x_1,\ldots,x_n]\backslash \{0\}\,|\,LM(u)=1 \}$$   Let $K[x_1,\ldots,x_n]_>:={S_>}^{-1}K[x_1,\ldots,x_n]=\{\frac{f}{u}\,|\,f,u \in K[x_1,\ldots,x_n], LM(u)=1\}$ \\the localization of $K[x_1,\ldots,x_n]$ with respect to $S_>$ and call $K[x_1,\ldots,x_n]_>$ the ring associated to $K[x_1,\ldots,x_n]$ and $>$.

\end{defn}

\begin{defn}
Let $>$ be any monomial ordering. For $f \in K[x_1,\ldots,x_n]_>$ choose $u \in K[x_1,\ldots,x_n]$ such that $LM(u)=1$ and  $uf \in K[x_1,\ldots,x_n]$. We define
$$LM(f)=LM(uf)$$
$$LC(f)=LC(uf)$$
$$LT(f)=LT(uf)$$
$$LE(f)=LE(uf)$$
$$tail(f)=f-LT(f)$$.

\end{defn}

\begin{defn}
Let $G$ be a subset  of $K[x_1,\ldots,x_n]$\\
\\ A $G$-monomial is a finite power product of the form $ G^\alpha=g_1^{\alpha_{1}}\ldots g_m^{\alpha_{m}}$ where $g_i\in G$ for $i=1,\ldots,m,$ and $\alpha=(\alpha_1,\ldots,\alpha_m)\in \mathbb{N}^m$. The set of all $G$-monomial is denoted by:$$Mon_G=\{G^\alpha| \ \alpha\in \mathbb{N}^m,\,m\in \mathbb{N} \}$$

Let $G=\{g_1,\ldots,g_s\}\subset K[x_1,\ldots,x_n]$ and $>$ be a local ordering. We define
 $$K[G]_>=(S_>\cap K[g_1,\ldots,g_s])^{-1}K[g_1,\ldots,g_s].$$

\end{defn}

\section{Sasbi basis of $K[[G]]$}
Let $\mathcal{M}$ is a maximal ideal of $K[[x_1,\ldots,x_n]]$. We fix a local degree ordering $>$ and use the notation of definition 1.2 which make sense in $K[[x_1,\ldots,x_n]]$ too. In this section we recall some result of \cite{HH1} and give an algorithm which computes Sasbi basis $K[[G]]$ upto a certain certain degree.

\begin{defn}
Given two elements $g,h \in K[[x_1,\ldots,x_n]]$, we will say that $g$ reduces to $h$ with respect to $G$ if their exist $G$-monomial $G^{\alpha}$ and $\gamma \in K$ such that $$h=g-\gamma G^{\alpha},\, \hbox{with}\,\, h=0 \, \,\hbox{or}\,\,LM(h)<LM(g)$$In this case we will write $$g \mathop\rightarrow\limits_{}^{G} h,$$ and we have that $g-h\in K[[G]]$.
\end{defn}

Consider a chain (possibly infinite) of reductions $$g \mathop\rightarrow\limits_{}^{G} h_1 \mathop\rightarrow\limits_{}^{G} h_2 \mathop\rightarrow\limits_{}^{G} \ldots \mathop\rightarrow\limits_{}^{G} h_m \mathop\rightarrow\limits_{}^{G} \ldots $$

This implies there exist $G$-products ${G}^{\alpha^{(i)}}$ and $a_{\alpha^{(i)}}\in K\backslash \{0\}$ such that $$h_m=g-\sum\limits_{i=1}^{m}a_{\alpha^{(i)}}{G}^{\alpha^{(i)}},$$ and because of the definition of the reduction $$LM(a_{\alpha^{(1)}}{G}^{\alpha^{(1)}})>LM(a_{\alpha^{(2)}}{G}^{\alpha^{(2)}})>\ldots $$If the chain is infinite, we get the following sequence in $K[[X]]:$ $$ s_m=\sum_{i=1}^{m}a_{\alpha^{(i)}}{G}^{\alpha^{(i)}},\,\,m\geq1.$$

This sequence happens to be convergent in $K[[x_1,\ldots,x_n]]$ with respect to the $\mathcal{M}$-adic topology. We denote the limit  of the sequence $(s_m)_{m\geq1}$ by $s$. Since all the terms are in the complete subalgebra $ K[[G]]$ so we have that $s\in K[[G]]$.
\begin{defn}
 If the reduction $h$ of $g$ is zero or for all $ x^{\beta}\in support(h),x^{\beta}\neq LM(G^{\alpha})$ for all G-monomial $G^{\alpha}$ then $h$ is called  $\mathbf{Normal \,form}$ of $g$ with respect to $G$. We denote the normal form $h$ by $NF(g\,|\,G)$.
\end{defn}
Normal forms always exist but the computation may use  infinite reductions. For computational reason we give an algorithm which computes the normal form up to the degree $d$.
\begin{alg}Let $>$ be any local degree ordering in $K[[x_1,\ldots,x_n]]$.
\\Input: $G\subset\mathcal{M} \backslash \{0\} $ , $g\in K[[x_1,\ldots,x_n]]$, $d\in \mathbb{Z}$.
\\Output: $h$=NF$(g\,|\,G,d)$ (the normal form of $g$ with respect to $G$ up to degree $d$)\footnote{for theoretical reasons we allow $G$ to be infinite and $d=\infty$. We have seen that for $d\rightarrow\infty$ the normal form $NF(g\,|\,G,d)$ converges in the $\mathcal{M }$-adic topology. We call this limit $NF(g\,|\,G)$, the normal form of $f$ with respect to $G$. }.
\begin{itemize}
\item $h:=g$;
\item while$(h\neq 0$ and $ord(h)\leq d)$\\
\indent$T_h=\{{G^{\alpha}} | \  LM( G^{\alpha})=LM(h)  \}\neq \phi $;\\
\indent \quad if $T_h\neq \phi$\\
\indent \qquad choose $G^{\alpha} \in T_h$ ;\\
\indent \qquad $h=h-\frac{LC(h)}{LC(G^{\alpha})} G^{\alpha}\,$;\\
\indent \quad else \\
\indent \qquad return $(LT(h)+NF(h-LT(h)\,|\,G,d);$
\item end(while)
\item return $h$;
\end{itemize}

\end{alg}
\begin{defn}
We say a set $G\subset \mathcal{M} \backslash \{0\}$ is $\mathbf{ Sasbi\, basis}$\footnote{In \cite{HH1} this is called a standard basis of subalgebras. We use this notation to be similar to sagbi bases introduced in \cite{RS1}.} of $K[[G]]$ if $$K[L(K[G])]=K[LM(G)]$$ where $L(K[G])=\{LM(g)|\,g\in K[[G]]\backslash\{0\}\,\}$, i.e $G$ that is a sasbi basis  if for all $f\in K[[G]]\backslash \{0\}$, $$LM(f)=LM(G^{\alpha})$$ for some $G$-monomial $G^{\alpha}$.

\end{defn}

\begin{exmp}
The set $G=\{ x^2,\sum_{i=3}^{\infty}x^{i} \} \subset K[[x_1,\ldots,x_n]]$ is a sasbi basis for  $K[[G]]$. Indeed, if $g\in K[[G]] \backslash \{0\}$, then $LM(g)=1$ or  $LM(g)=x^{\alpha}$, for some $\alpha \geq 2$. Hence $LM(g)\in K[x^2,x^3]=K[LM(G)]$.
\end{exmp}

Now for the characterization of the sasbi bases  in $K[[x_1,\ldots,x_n]]$ similar to those  in $K[x_1,\ldots,x_n]$ we need to define an analogue of the S-polynomial.

\begin{defn}Let $G=\{g_1,\ldots,g_s\}\subset \mathcal{M} \backslash \{0\}$. An
 $\mathbf{S}$-$\mathbf{polynomial}$ is an element of the form $$a\,{G}^{\alpha}-b\,{G}^{\beta}$$ where $a,b\in K \backslash \{0\}$ and $G^{\alpha}$ and $G^{\beta}$ are $G$-monomials, such that $LT(a\,G^{\alpha})=LM(b\,G^{\beta})$.
\end{defn}
Next theorem gives criteria for a set to be a sasbi basis  of $K[[G]]$.
\begin{thm}
(c.f. \cite{HH1}, page 50) Given $G=\{g_1,\ldots,g_s\}\subset \mathcal{M}\backslash \{0\} $, $G$ is a sasbi basis of  $K[[G]]$ if and only if every S-polynomial of $G$ has a vanishing normal form with respect to $G$.
\end{thm}
The following is the analogue of Buchberger's Algorithm for subalgebras in $K[[x_1,\ldots,x_n]]$.
\begin{alg}$\,$\\
Input: A finite subset $G$ of $\mathcal{M}\backslash \{0\}$.
\\Output: A sasbi bases $F$ for $K[[$G$]]$.
\begin{itemize}
\item $F=G$ ;
 \item $old F=\phi$ ;
\item while ($F\neq old F$)\\
\indent $S:=\{s \,|\, s \,\hbox{ is a}\, \hbox{ $S$-polynomial of}\, F \}$;\\
\indent $R:=\{r \,|\,r=NF(s\in S\,|\,F) \,\hbox {and}\, r\neq 0 \};$\\
\indent $old{F}={F}$;\\
\indent ${F}={F}\cup R$;
\item return ${F}$;
\end{itemize}
\end{alg}

\section{Sasbi Basis in the Localization of $K[x_1,\ldots,x_n]$}

   In this section first of all we will introduce Sasbi bases in $K[G]_{>}$ and prove that Sasbi bases in $K[G]_{>}$ are also Sasbi bases in $K[[G]]$.  We will prove that also in the general case the computation of a Sasbi basis with respect to a local ordering can be reduced using homogenization to the computation of a Sagbi basis with respect to a suitable global ordering. This is also here a very expensive way to compute a Sasbi basis. Therefore later a more efficient algorithm is presented. We introduce notion of weak sasbi normal form of a polynomial with respect to $G$ in $K[x_1,\ldots,x_n]_>$ and give an algorithm to compute it. Then we give a criterion for a set to be a Sasbi basis, which is the base of an algorithm to compute the Sasbi basis.

Let ${G}=\{g_1,\ldots,g_s \}\subset K[x_1,\ldots,x_n]$, $\mathcal{G}=K[G]_>$
 and $L(\mathcal{G})=\{LM(g) \,|\,g\in \mathcal{ G}\backslash\{0\}\}$.
\begin{defn}
A subset  $S\subset{\mathcal{G}}$ is called \textbf{SASBI\footnote{SASBI stands for ``Subalgebra Analogue to Standard Basis For Ideal" } Basis  } of $K[G]_>$ if
$$K[L(\mathcal{G})]=K[L(S)]$$ i.e for all $g\in K[G]_>\backslash\{0\}$ $$LM(g)=LM(S^{\alpha})$$ for some S-monomial $S^{\alpha}$.
\\ If $>$ is global, a Sasbi basis is also called a Sagbi basis.
\\ If we just say that $S$ is a Sasbi basis, we mean that $S$ is a Sasbi basis of the  $K[S]_{>}$ generated by $S$.
\end{defn}


\begin{thm}
Let $K[x_1,\ldots,x_n]_>\subset K[[x_1,\ldots,x_n]]$ be equipped with  local degree ordering $>$. Let $G=\{g_1,\ldots,g_s \}$ be a subset of $K[x_1,\ldots,x_n]$. If $S$ is a Sasbi basis of $K[G]_>$ then  $S$ is a Sasbi basis of $K[[G]]$.
\end{thm}
\begin{proof}
For $g\in K[[G]]$ we have to prove that there exist $G$-monomial $G^{\alpha}$ such that $LM(g)$=$LM(G^{\alpha})$. If $g\in K[[G]]$  there exist $H\in K[[y_1,\ldots,y_s]]$ such that we have $g=H(g_1,\ldots,g_s)$. There  exist a decomposition of  $H=H^{(0)}+H^{(1)}$, $H^{(0)}\in K[y_1,\ldots,y_s]$ and $H^{(1)}\in K[[y_1,\ldots,y_s]]$ such that $$LM(H^{(0)}(g_1,\ldots,g_s))=LM(H(g_1,\ldots,g_s))=LM(g).$$ Since $S$ is Sasbi bases of $K[G]_>$  there exist a $G$-monomial $G^{\alpha}$, such that $LM(G^{\alpha})=LM(H^{(0)}(g_1,\ldots,g_s)).$  We get $ LM(g)=LM(G^{\alpha})$  which shows that $S$ is a Sasbi bases  for $K[[G]]$.
\end{proof}

 Now we want to show how to reduce the SASBI bases computation for local orderings using homogenization with respect to a variable ``$t$" to the computation of SASBI bases for global orderings.
\begin{thm}
Let $H=\{g_1,\ldots,g_m\}\subset K[x_1,\ldots,x_n]$ and $K[H]=K[g_1,\ldots,g_m]_>$. Here $>$ is a local  monomial ordering given by a matrix $M$. Consider $K[t,x_1,\ldots,x_n]$ with monomial ordering $>_h$ defined by the matrix
\(
\left( {\begin{array}{*{20}c}
   1 & 1 &  \cdots  & 1  \\
   0 & {} & {} & {}  \\
    \vdots  & {} & M & {}  \\
   0 & {} & {} & {}  \\
\end{array}} \right)
\)
\\ $>_h$ is a global ordering.
 We define $G_i$ to be $\,G_i:=g_i^h \in K[t,x_1,\ldots,x_n]$. Assume $\widehat{S}= \{S_1,S_2,\ldots,S_k \}\subset K[t,x_1,x_2,\ldots,x_n] $ is a  Sagbi basis of
$K[G_1,\ldots,G_m]$ with respect to $>_h$. Let $s_j:=S_j(t=1)\, ,1\leq j\leq k$,
then $S=\{s_1,\ldots,s_k \}$ is a Sasbi basis of  $K[g_1,g_2,\ldots,g_m]_>$.
\end{thm}

\begin{proof}
 We want to show $S$ is a Sasbi basis for $K[H]_>$. For this we have to show that
\\1. $S \subset K[H]_>$.
\\2. For $g\in K[H]_>$ there exist $ \alpha =( \alpha_1,\alpha_2,\ldots,\alpha_k)\in \mathbb{N}^k$ such that $LM(g)=LM(S^\alpha)$.
\\1) We know that $\widehat{S}=\{S_1,\ldots,S_k \}$ is a Sagbi basis of  $K[G_1,\ldots,G_m]$ so $S_i=\sum \gamma_{i,j}G^{\alpha_{i,j}}$ with $\gamma_{i,j} \in K $.
Put $t=1$ we get $s_i=\sum \gamma _{i,j}H^{\alpha_{i,j}}$ this implies $s_i\in K[H]_{>}$.
\\2)For $g \in K[H]_>$  there exists $u\in S_> \cap K[H]$ such that $u.g=\sum \gamma _j H^{\alpha_{j}}$, then there exists $\rho \in \mathbb{Z}$ such that
$t^{\rho}.u^{h}.g^{h}=\sum \gamma  _i (H^{\alpha_{j}})^h=\sum \gamma _{j} G^{\alpha_{j}}$. We have  that
$\widehat{S}$ is Sagbi basis of $K[G_1,\ldots,G_m]$. Then there exists $\alpha =( \alpha_1,\alpha_2,\ldots,\alpha_k )\in \mathbb{N}^{k}$ such that $LM(t^{\rho}.u^{h}.g^{h})=LM(S^{\alpha})$. Since $LM(G)|_{t=1}=LM(G|_{t=1})$,therefore
  $LM(t^{\rho}.u^{h}.g^{h})|_{t=1} =LM(g)$, since $LM(u)=1$ as $u\in S_>$ and  $LM(\widehat{S}^{\alpha})_{t=1}=LM(S^\alpha)$,
we obtain $LM(g)=LM(S^\alpha)$.
\end{proof}

Theorem $3.3$  shows that Sasbi bases are computable in many cases. It turns out that similar to the theory of standard bases with respect to local orderings for ideals this approach is not very efficient. Therefore one should like to have an efficient way for computing Sasbi bases. The basis for this is the concept of the normal form.

\begin{defn}
Let $G$ and $g$ be a finite subset and a polynomial in $K[x_1,\ldots,x_n]$ respectively,such that $K[G]_{>}$ admits a finite Sasbi bases and $g\in K[G]$. \\We say that a polynomial $h$ is a \textbf{Weak SASBI normal form} of $g$ with respect to $G$, and we write $h=SNF(g|G)$, if\\
0. $h=SNF(0| \ G)=0$\\
1. $h\neq 0 \Rightarrow$ $LM(h)  \notin K[LM(G)]$
\\2. There exist unit $u\in S_{>}\cap K[G]$ such that $ug-h$ has a  representation with respect to $G$, that is either $ug-h=0$ or $ug-h=\sum_{i=1}^{v }\gamma_{i}G^{\alpha_{i}}$ where $\gamma_{i}\in K$ and $LM(g)= {\max}_{i=1}^{v} \{LM(\gamma_{i}G^{\alpha_{i}})\}$.
This representation is called $\textbf{SASBI representation}$.

\end{defn}
\begin{alg}$\,$\\
Input: $f$, $G$, $>$ a local monomial ordering. We assume that $G=\{g_1,\ldots,g_s\}$ and $f$ are subset and polynomial in $K[x_1,\ldots,x_n]$ such that $f\in K[G]$. We also assume there  exist a finite Sagbi basis of $K[H]$ where $H=G^h$ the homogenization of $G$ with respect to $``t"$, a new variable.
\\Output: $h\in K[x]$ a polynomial weak Sasbi normal form of $f$ with respect to $G$.
\begin{itemize}
\item $h:=f$
\item $T:=G$
\item while$(h\neq 0$ and
$T_h=\{{T^{\alpha}},\,T\hbox{-monomial} | \  LM( T^{\alpha})=LM(h)  \}\neq \phi $\\
\indent choose $T^{\alpha} \in T_h$ such that $ecart(T^{\alpha})$ is minimal;\\
\indent if $ecart(T^{\alpha}) > ecart(h)$\\
 \indent \quad   $T:=T \cup \{h\}$;\\
\indent $h=h-\frac{LC(h)}{LC(H^{\alpha})} T^{\alpha}\, \hbox{for some}\, \gamma\in K ; $
\item return $h$;
\end{itemize}
\end{alg}

\begin{proof}
Termination is most easily seen by using homogenization: start with $h:=f^h$ and $ H:=G^h=\{g^h |\,g\in G \}$.
  \\The while loop looks as follows
 \begin{itemize}
\item $\hbox{while}(h \neq 0$ and $T_h=\{H^{\alpha},\,H\hbox{-monomial}\,  |\, LM(H^{\alpha})=t^{\beta }LM(h)\hbox{ for some } \beta  \}$)\\
\indent choose $g\in T_h$ in a way with $\beta \geq 0$ is minimal;\\
\indent if $\beta > 0$\\
\indent  \quad $T=T \cup \{h\}$;\\
\indent $h:=h-\frac{LC(h)}{LC(H^{\alpha}} H^{\alpha} $;\\
\indent $h:=(h|_{t=1})^h$;
\end{itemize}
By our assumption $K[H]$ has a finite sagbi bases, there exists some positive integer $N$ such that $K[L(T_v)]$ becomes stable for $v \geq N$, where $T_v$ denotes the set $T$ after the $v$-th turn of the while loop. The next $h$, satisfies $LM(h)\in K[L(T_N)]=K[L(H)]$, whence  $LM(h)=LM(H^{\alpha})$ for some $H^{\alpha}\in K[H]$  and $\beta =0$, that is, $T_v$ itself becomes stable for $v \geq N$ and the algorithm continues with fixed $T$. Then it terminates, since $>$ is a well ordering on $K[t,x]$.

 To see the correctness, consider the $i$-th turn in the while loop of algorithm.
\\There we create  $T_i= \{g_1,g_2,...,g_s,h_0,h_1,\ldots,h_{i-2}\}$ such that $h_i=h_{i-1} - \gamma^i{T^{\alpha^{(i)}}}$ and $LM(T^{\alpha^{(i)}})=LM(h_{i-1})>LM(h_i)$ where $T^{\alpha^{(i)}}$ is $T_i$- monomial.
\\Suppose, by induction, that in the first $i-1$ steps  we have constructed SASBI representations
$$u_jf=\sum\limits_{l=1}^{v^{(j)}}{\gamma_{l}^{(j)}}{G^{\alpha_{l}^{(j)}}}+h_j \  \hbox{where} \  {\gamma_{l}^{j}}\in K \  \hbox{and} \  LM(f)= {\max}_{l=1}^{v^{(j)}}\{LM({\gamma_{l}^{(j)}}{G^{\alpha_{l}^{(j)}}})\}.$$                          \\ where $u_j \in S_{>} \cap K[G]$ and   $1\leq j \leq{i-1}$
\\ We have to prove $ \ \exists u_i\in  S_{>} \cap K[G]$ and $u_if=\sum\limits_{l=1}^{v^{(i)}}{\gamma_{l}^{(i)}}{G^{\alpha_{l}^{(i)}}}+ h_i$ and $ LM(f)= {\max}_{l=1}^{v^{(i)}}\{LM({\gamma_{l}^{(i)}}{G^{\alpha_{l}^{(i)}}})\}.$
\\
\\We have two possibilities
\\$1)$ $T^{\alpha^{(i)}}=G^{\alpha^{(i)}}$ is $G$-monomial.
\\$2)$ $T^{\alpha^{(i)}}=T_i^{\alpha^{(i)}}$ is $T_i$-monomial.
\\  Induction step: Consider the SASBI representation for $j=i-1$.
$$u_{i-1}f=\sum\limits_{l=1}^{v^{(i-1)}}{\gamma_{l}^{(i-1)}}{G^{\alpha_{l}^{(i-1)}}}+h_{i-1} \,\,\,\,\, \hbox{and}\,\,\,\,\, LM(f)= {\max}_{l=1}^{v^{(i-1)}}\{LM({\gamma_{l}^{(i-1)}}{G^{\alpha_{l}^{(i-1)}}})\}.$$
 \\ For the first case in induction step , replace $h_{i-1}$ by $ \gamma^i{G^{\alpha^{(i)}}}+h_i $ ,and obtain
 $$u_{i-1}f=\sum\limits_{l=1}^{v^{(i-1)}}{\gamma_{l}^{(i-1)}}{G^{\alpha_{l}^{(i-1)}}}+\gamma^i{G^{\alpha^{(i)}}}+h_i.$$
 \\ Put $u_i=u_{i-1}$ and ${\gamma_{l}^{(i-1)}}{G^{\alpha_{l}^{(i-1)}}}={\gamma_{l}^{(i)}}{G^{\alpha_{l}^{(i)}}} \ 1\leq l \leq v^{(i-1)},\,\gamma^i{G^{\alpha^{(i)}}}={\gamma_{v_i}^{(i)}}{G^{\alpha_{v_i}^{(i)}}}$ we get the required representation
 $$u_{i}f=\sum\limits_{l=1}^{v^{(i)}}{\gamma_{l}^{(i)}}{G^{\alpha_{l}^{(i)}}}+h_i.$$
 As  $LM({\gamma_{v_i}^{(i)}}{G^{\alpha_{v_i}^{(i)}}})< LM({\gamma_{l}^{(i)}}{G^{\alpha_{l}^{(i)}}}),\,  1\leq l \leq v_i-1$, from this condition we get $LM(f)= {\max}_{l=1}^{v^{(i)}}\{LM({\gamma_{l}^{(i)}}{G^{\alpha_{l}^{(i)}}})\}$ which shows representation is Sasbi.
 \\For the second case in induction step  replace the $h_{i-1}$ by $ \gamma^i{T_i^{\alpha^{(i)}}}+h_i $ , it becomes
 $$u_{i-1}f=\sum\limits_{l=1}^{v^{(i-1)}}{\gamma_{l}^{(i-1)}}{G^{\alpha_{l}^{(i-1)}}}+\gamma^i{T_i^{\alpha^{(i)}}}+h_i.$$
 We can write ${T_i^{\alpha^{(i)}}}={G^{\beta^{(i)}}}{H^{\gamma^{(i)}}}$ where $H=\{h_0,h_1,\ldots,h_{i-2}\}$. Since we are in second case  not all the components of $\gamma^{(i)}$ are zero. Since $LM(h_{i-1})<LM(h_j)$ for $j\leq i-2$ and $LM(T^{\alpha(i)})=LM(h_{i-1})<LM(f)$ it follows that $LM(G^{\beta(i)}<1$. Since $h_j=u_jf-\sum\limits_{l=1}^{v^{(j)}}{\gamma_{l}^{(j)}}{G^{\alpha_{l}^{(j)}}}$,$1\leq j \leq{i-2}$ , we can replace $h_j's$ by this expression therefore
 $$ {T_i}^{\alpha(i)}=G^{\beta(i)}R(u_0,u_1,\ldots,u_{i-2},f,g_1,\ldots,g_s)f+G^{\beta(i)}L^{\gamma(i)}.$$
 For a suitable polynomial $R$ and $L=\{\,\sum\limits_{l=1}^{v^{(0)}}{\gamma_{l}^{(0)}}{G^{\alpha_{l}^{(0)}}},\ldots,\sum\limits_{l=1}^{v^{(i-2)}}{\gamma_{l}^{(i-2)}}{G^{\alpha_{l}^{(i-2)}}}   \, \}   $.

  Since $u_0,u_1,\ldots,u_{i-2},f\in K[G] $ and $LM(G^{\beta(i)})<1$ it follows that
 $$u_i=u_{i-1}-\gamma^{i}G^{\beta(i)}R\in S_>\cap K[G].$$
 Since $LM({\gamma_{l}^{(j)}}{G^{\alpha_{l}^{(j)}}})\leq LM(f) $ it follows that leading monomial of any $G$-monomial occuring in $G^{\beta(i)}L^{\gamma(i)}$ is smaller than the leading monomial of $f$. This implies
  $$u_{i}f=\sum\limits_{l=1}^{v^{(i-1)}}{\gamma_{l}^{(i-1)}}{G^{\alpha_{l}^{(i-1)}}}+\gamma^i{G^{\alpha^{(i)}}}L^{\gamma(i)}+h_i.   $$ is a sasbi representation since $LM(f)={\max}_{i=1}^{v^{(i-1)}} \{LM({\gamma_{l}^{(i-1)}}{G^{\alpha_{l}^{(i-1)}}})\}.  $
 \end{proof}

 \begin{exmp}In the localization of the univariate polynomial ring $K[x]_{>}$ where $>$ is the local ordering take $g=x^3+x^4$ and $G=\{g_1=x^3+x^6, g_2=x-x^{2} \}$ we want to compute the weak Sasbi normal form of $g$ with respect to $G$.
 \\In the first reduction we select the G-monomial $g_1=x^3+x^6$ with minimal ecart such that $LM(h_0)=LM(f)=LM(g)=x^3$ , we have ecart$(h_0)$=1, ecart$(g_1)=3$, so ecart$(g_1)>$ecart$(h_0)$ therefore  we have to enlarge $G=\{g_1=x^3+x^6, g_2=x-x^{2}, g_3=x^3+x^4 \}$   and $$h_1=h_0-g_1$$ $$x^4-x^6=x^3+x^4-(x^3+x^6)$$
 \\In the second reduction we select the $G$-product $g_2g_3=(x-x^{2})(x^3+x^{4})=x^4-x^6$ with minimal ecart such that $LM(h_1)=LM(g_2g_3)=x^4$. Now we have ecart$(g_2g_3)=2$, ecart$(h_1)=2$ so $G$ remains the same
   and $$h_2=h_1-g_2g_3 ,$$
$$0= x^4-x^6-(x^4-x^6) ,$$
we get $h_2=0$. Now we summarize and obtain
$$h_2=h_1-g_2g_3 ,$$ $$0=x^4-x^6-(x-x^2)(x^3+x^4).$$
As $h_1=g-g_1$
$$h_2=g-g_1-g_2g_3 ,$$ $$0=x^3+x^4-(x^3-x^6)-(x-x^2)(x^3+x^4),$$
 $g_3=g$ we get
$$g-gg_2=g_1+h_2 ,$$ $$x^3+x^4-(x-x^2)(x^3+x^4)=x^3-x^6 ,$$
$$(1-g_2)g=g_1+h_2$$ $$(1-x-x^2)(x^3+x^4)=x^3-x^6 .$$
we have $1-g_2=1-x-x^2\in S_>\cap K[G]$ and $h_2=0$ is the weak Sasbi normal form.

 \end{exmp}



\begin{defn}
 Let $G=\{g_1,\ldots,g_m\}\subset K[x_1,\ldots,x_n]$. Let   $$AR(G):=\{h\in K[y_1,\ldots,y_m]\,|\,h(LM(g_1),\ldots,LM(g_m))=0\}\subset K[y_1,\ldots,y_m].$$

\end{defn}
\begin{defn} Let $G\subseteq K[x_{1},...,x_{n}]$ and $\sum^{\nu}_{i=1}\gamma_{i}G^{\alpha_{i}}\in K[G] $. We define the height  $ht(\sum^{\nu}_{i=1}\gamma_{i}G^{\alpha_{i}})= max_{i=1}^{\nu}\{LM(G^{\alpha{i}})\}$.
\end{defn}
\begin{thm} (SASBI basis criterion) Let $G$={$\{g_{1},g_{2},\ldots ,g_{m}\}$} be a subset of\\$ K[x_{1},...,x_{n}]$. Assume that $K[H]$ has a finite sagbi basis where $H=\{{g_1}^h,\ldots,{g_m}^h\}$, the homogenization of $G$. Let   $\mathcal{S}:=\{P_{1},...,P_{k}\}$ be a generating set\footnote{The set of $S$-polynomials defined in definition 2.6 defines a generating set of $AR(G)$.} of  $AR(G)$. Then $G$ is a SASBI basis for $K[G]_{>}$ if and only if for each $1\leq j \leq k $, SNF($P_{j}$(G)$\mid G$)=0.
\end{thm}
 \begin{proof} $(\Rightarrow)$ Suppose  that  $SNF(P_{j}(G)\mid G)\neq0$. This implies that $LM(SNF(P_{j}(G)|\\ G)\notin K[LM(G)]$ by the property of the weak Sasbi normal form . We have $P_j(G)\in K[G]$ therefore $SNF(P_{j}(G)\mid G)\in K[G]$. Since $G$ is a SASBI basis of $K[G]_>$ we have $LM(SNF(P_{j}(G)\mid G)\in K[LM(G)]$. This is contradiction to the assumption that $LM(SNF(P_{j}(G)\mid G)\notin K[LM(G)]$.
  \\$(\Leftarrow)$ To prove that $G$ is SASBI basis , we have to prove  that $g\in K[G]_> $ has a SASBI representation with respect to $G$, that is there exist $u\in S_>\cap K[G]$ such that
   $$ug=\sum_{i=1}^{\nu}\gamma_{i}G^{\alpha_{i}} \,\, with\,\,  LM(g)=ht(\sum_{i=1}^{\nu}\gamma_{i}G^{\alpha_{i}})$$

Let $g\in K[G]_{>}$, choose $u\in S_>\cap K[G]$ such that  $ug=\sum_{i=1}^{\nu}\gamma_{i}G^{\alpha_{i}}$, furthermore, we assume that this representation has the smallest possible height of all possible representations of $ug$ in $K[G]$. We denote this height by X:=$max_{i=1}^{\nu}\{LM(G^{\alpha_{i}})\}$. It is clear that $LM(g)\leq X$.
 Suppose that $LM(g)\lvertneqq X$. Without loss of generality, let the first $\mu$ summands in the above representation of $g$, be the ones for which  X=$LM(G^{\alpha_{i}})$. Then cancelation of their leading terms must occur, that is, $\sum_{i=1}^{\mu}\gamma_{i}LT(G^{\alpha_{i}})=0$, and hence we obtain a polynomial in $K[y_{1},...,y_{m}]$,\,$P(y)=\sum_{i=1}^{\mu}\gamma_{i}y^{\alpha_{i}}\in AR(G)$. Since,  $\mathcal{S}=\{P_{1},...,P_{k}\}$ is a generating set of $AR(G)$ we can write
 $$\hspace{4.4 cm}P(y)=\sum_{j=1}^{k}f_{j}P_{j}(y)\hspace{4 cm}(* ) $$
  For suitable $f_j\in K[y_1,\ldots,y_m]$. Furthermore , note that \\ $$ht(P(G))=max_{j=1}^{k}ht(g_{j}(G))ht(P_j(G))=X  $$ \\
  where, $f_{j}(G) \ and\ P_{j}(G) \ $ are considered as expressions in the $g_{i}\ 's$ \\
  On the other hand :
   \\ By assumption we have   for all $1\leq j \leq k \ ,\, SNF(P_{j}(G)\mid G)=0$, which means that $w_jP_{j}(G)$  has a SASBI representation,  $w_jP_{j}(G)=\sum_{l=1}^{\nu_{j}}\gamma_{l_{j}}G^{\alpha_{l_{j}}}$, for suitable $w_j\in S_>\cap K[G]$ and   $LM(P_{j}(G))=max_{l=1}^{\nu_{j}}\{LM(G^{\alpha_{l_{j}}})\}\lvertneqq ht(P_{j}(G))$. The inequality is strict since $P_{j}\in AR(G)$, we may assume that $w=w_j  ,\,\hbox{where}\,\, 1\leq j \leq k$ . For each $j$, we have

 $$ \hspace{3.75 cm}wf_{j}(G)P_{j}(G)=\sum_{i=1}^{\nu_{j}}\gamma_{l_{j}}g_{j}(G)G^\alpha_{l_{j}}.\hspace{3 cm}(**)$$

If we define $X_{j}$ to be the height of the right hand side in the equation , then obtain
 $$ X_{j}\lvertneqq max_{j=1}^{k}ht(f_{j}(G).ht(P_{j}(G))=X . $$
Finally, the equations $(*)$ and $(**)$  imply that :
 $$ug=P(G)+\sum_{i=\mu+1}^{\nu}\gamma_{i}G^{\alpha_{i}}$$
 $$= \underbrace{\sum_{j=1}^{k}\sum_{l=1}^{\nu_j}\gamma_{j_{j}}f_{j}(G)G^{\alpha_{l_{j}}}}_{sum _1}+\underbrace{\sum_{i=\mu+1}^{\nu}\gamma_{i}G^{\alpha_{i}}}_{sum_2}.$$
If we examine the expressions of the above equation, we see that  $ \, X_j<X \,;\,
\hbox {for all}\,\, 1\leq j\leq k \,\,\hbox{therefore}\,\, ht(sum_{1})=max_{j=1}^{k}X_{j}<X.$  By the choice of   $ \mu,\ ht(sum_{2}< X).$ \\ But this contradicts our  assumption that we have chosen a representation of $h$ with  smallest possible height. Thus, $G$ is a SASBI basis of $K[G]_{>}$.
\end{proof}
This theorem is the base of following algorithm :

\begin{alg}
  Let $>$ be a local monomial ordering on $K[x_{1},....,x_{n}]$.
\\Input: A finite subset ${G}\subset {K}[x_{1},....,x_{n}]$. Assume $K[G]_{>}$ admits a finite SASBI basis and $K[H]$ admits a finite sagbi basis where $H=G^h$ is the homogenization with respect to new variable ``t".\\
Output: A SASBI basis ${F}$ for ${K}[{G}]_>$.\\
\begin{itemize}
\item $F=G$;
 \item $old F=\phi$;
\item While ($F\neq old F$)\\
\indent  Compute a generating set $\mathcal{S}$ for $AR({F})$;\\
\indent $\mathcal{P}=\mathcal{S}({F})$; \\
\indent $Red$=\{$\textrm{SNF}(p\mid {F})\mid p\in \mathcal{P}\setminus\{0\}\}$;\\
\indent $old{F}={F}$;\\
\indent ${F}={F}\cup Red$;\\
\item return ${F}$;
\end{itemize}
\end{alg}
\begin{exmp}
Let $G=\{g_1=x^4,g_2=x4+x5+x6, g_3=y^2, g_4=x^7, g_5=y^3+x^8\}$ is a subset $K[x,y]$ and $>$ the degree lexicographical  local monomial ordering. We consider $K[G]_>=K[x^4,x4+x5+x^6, y^2, x^7, y^3+x^8]_>$ . Then we have   an ideal $AR(G)(G)=(s_1=x^8y^3+\frac{1}{2}x^{16}, s_1=x^5+x^6$). We can take the reduction of $s_1$ by ${g_1}^2g_5$ (with minimal ecart) we obtain $$h=s_1-{g_1}^2g_5$$ $$x^8y^3+\frac{1}{2}x^{16}-{(x^4)}^2(y3+x^8)=0$$ so $SNF(s_1\,|\,G)=0$. There is no $G$-monomial $G^{\alpha}$ such that $LM(G^{\alpha})=LM(s_2)=x^5$, so $SNF(s_2\,|\,G)=x^5+x^6$. We have new
$G=G\cup\{ g_6=x^5+x^6\}$. Then we have new $AR(G)(G)=(s_1=x^8y^3+\frac{1}{2}x^{16}, s_1=x^5+x^6$), so obviously  weak sasbi normal form of $s_1$ and $s_2$ are $0$. This shows that $G=\{g_1,g_2, g_3, g_4, g_5, g_6\}$ is a Sasbi basis of $K[x^4,x4+x5+x^6, y^2, x^7, y^3+x^8]_>$.
\end{exmp}

We have presented the theory of sasbi basis for $K[G]_>$, where $G$ is finite subset of $K[x_1,\ldots,x_n]$ and $>$ is local orderings, but it is still an open problem for  mixed orderings.
\section{Implementation in SINGULAR  }
In this section we will give an overview of the main procedures which we have implemented in SINGULAR. In this overview we will present these procedures and give by concrete SINGULAR examples to explain their usage. We have implemented three types of procedures:
\\

1) \underline{\textbf{Weak sasbi Normal form procedure}}\\

 " WSNF procedure": It is an  implementation of Algorithm 3 (ecart driven normal form) to obtain weak sasbi normal form of a polynomial.\\

\textbf{SINGULAR Procedure}:\\

\texttt{LIB"algebra.lib"}  ;// \texttt{we need this library for "algebra\_containment"
\\ \indent \indent \qquad \qquad \qquad \qquad // procedure}    \\

\texttt{proc WSNF(poly f,ideal I)}\\
\indent \texttt{\{ }\\
\indent\quad \texttt{ideal G=I} ;\\
\indent \quad \texttt{poly h=f} ;\\
\indent\quad \texttt{poly h1,j };\\
\indent \quad \texttt{list L} ;\\
\indent \quad \texttt{map psi ;}\\
\indent \quad \texttt{while(h!=0 \&\& h1!=h)}\\
\indent  \quad\texttt{\{}\\
\indent \quad\quad   \texttt{L= algebra\_containment(lead(h),lead(G),1)} ;\\
\indent \quad\quad  \texttt{if (L[1]==1)}\\
\indent \quad\quad \quad  \texttt{\{}\\
\indent \quad\quad \quad \quad  \texttt{def s= L[2]} ;\\
\indent \quad\quad \quad \quad  \texttt{psi= s,maxideal(1),G };\\
\indent \quad\quad \quad \quad \texttt{ j= psi(check)} ;\\
\indent \quad\quad \quad \quad  \texttt{if (ecart(h)$<$ecart(j))}\\
\indent \quad\quad \quad \quad \texttt{\{}\\
\indent \quad\quad \quad \quad \quad  \texttt{G[size(G)+1]=h };\\
\indent \quad\quad \quad \quad \texttt{\}}\\
\indent \quad\quad \quad \quad  \texttt{h1=h };\\
\indent \quad\quad \quad \quad  \texttt{h=h-j };\\
\indent \quad\quad \quad \quad  \texttt{kill s };\\
\indent \quad\quad \quad  \}\\
\indent \quad\texttt{\}}\\
\indent \quad\texttt{ return (h)} ;\\
\indent\texttt{\}}\\

\textbf{SINGULAR Example 4.1 }\\

 \texttt{ring r=0, (x,y), Ds };\\
 \indent\texttt{ideal i=x2,x4+x5+x6,x7,y2,y3+x8 };\\
 \indent\texttt{poly f=x4y3+y5} ;\\
 \indent\texttt{WSNF(f, i) };\\
 \indent$=>$ \texttt{x5y3-x6y3-x8y2-x12-x13-x14} \\

 \indent\texttt{ring r=0, (x), ls }; // example 3.6\\
 \indent\texttt{ideal i=x3+x4} ;\\
 \indent\texttt{poly g=x3+x6, x-x2;} \\
 \indent\texttt{WSNF(g, i) };\\
 \indent\texttt{$=>$ 0} \\

2) \underline{\textbf{Procedure to compute S-polynomials}}\\

"sasbiSpoly procedure": This procedure computes the generators of $AR(G)$
(defined in definition 3.10) which are S-polynomials.\\

\textbf{SINGULAR Procedure}:\\

\texttt{LIB"elim.lib" ;$\,\,$// we need this library for "nselect" procedure}

\texttt{proc sasbiSpoly(ideal id)}\\
\indent \{\\
\indent \quad \texttt{def bsr= basering };\\
 \indent \quad  \texttt{ideal vars = maxideal(1)} ;\\
 \indent \quad \texttt{int n=nvars(bsr)} ;\\
 \indent \quad \texttt{int m=ncols(id) };\\
 \indent \quad \texttt{int z} ;\\
 \indent \quad \texttt{ideal p };\\
 \indent \quad \texttt{if(id==0)}\\
 \indent\quad\texttt{\{}\\
 \indent \quad \quad \texttt{return(p) };\\
 \indent\quad \}\\
\indent\quad \texttt{else}\\
 \indent\quad\{\\
   \indent \quad \quad \texttt{execute("ring R1=("+charstr(bsr)+"),(@y(1..m),"+varstr(bsr)+"),\\
   \indent \qquad \quad \quad \qquad  (ds(m),ds(n));");}\\
   \indent \quad \quad \texttt{ideal id =imap(bsr,id)} ;\\
   \indent \quad \quad \texttt{ideal A };\\
   \indent \quad \quad \texttt{for (z=1; z$<$=m; z++)}\\
   \indent \quad \quad \texttt{\{}\\
     \indent \quad \quad \quad \texttt{A[z]=lead(id[z])-@y(z) };\\
   \indent \quad \quad  \}\\
   \indent \quad \quad \texttt{A=std(A)} ;\\
   \indent \quad \quad \texttt{ideal kern=nselect(A,m+1,m+n) };\\
   \indent \quad \quad \texttt{export kern,A} ;\\
   \indent \quad \quad\texttt{ setring bsr} ;\\
  \indent \quad \quad \texttt{map phi= R1,id };\\
   \indent \quad \quad \texttt{p=simplify(phi(kern),1)} ;\\
   \indent \quad \quad \texttt{return (p)} ;\\
 \indent \quad \}\\
\indent\}\\

\textbf{SINGULAR Example 4.2. }\\

 \texttt{ring r=0, (x,y), Ds} ;\\
 \indent\texttt{ideal i=x2,x4+x5+x6,x7,y2,y3+x8 };\\
 \indent\texttt{sasbiSpoly(i)};\\
 \indent\texttt{[1]=x5+x6} \\
 \indent\texttt{[2]=x8y3+1/2x16 }\\

3) \underline{\textbf{SASBI BASIS construction algorithm}}\\

 "Sasbi procedure": It is an iterative consequence of previous procedures to compute sasbi basis.\\

\textbf{SINGULAR Procedure}:\\

\texttt{proc Sasbi(ideal id)}\\
\indent\{\\
 \indent \quad \texttt{ideal S,oldS,Red };\\
 \indent \quad \texttt{list L} ;\\
 \indent \quad \texttt{int z,n };\\
 \indent \quad \texttt{S=id };\\
 \indent\quad \texttt{while( size(S)!=size(oldS))}\\
 \indent \quad \{\\
  \indent \quad \quad  \texttt{L=sasbiSpoly(S)} ;\\
  \indent \quad \quad  \texttt{n=size(L) };\\
   \indent \quad \quad \texttt{for (z=1; z$<$=n; z++)}\\
   \indent \quad \quad \{\\
     \indent \quad \quad \quad \texttt{Red=L[1][z] };\\
     \indent \quad \quad \quad \texttt{Red=WSNF(Red[1],S) };\\
     \indent \quad \quad \quad \texttt{oldS=S} ;\\
     \indent \quad \quad \quad \texttt{S=S+Red} ;\\
   \indent \quad \quad\}\\
 \indent \quad \}\\
\indent \quad \texttt{return(S) };\\
\indent\}\\

\textbf{SINGULAR Example 4.3. }\\

 \texttt{ring r=0, (x,y), Ds };\\
 \indent\texttt{ideal i=x2,x4+x5+x6,x7,y2,y3+x8} ;\\
 \indent\texttt{Sasbi(i)};\\
 \indent\texttt{[1]=x2}\\
 \indent\texttt{[2]=x4+x5+x6}\\
 \indent\texttt{[3]=x7}\\
 \indent\texttt{[4]=y2}\\
 \indent\texttt{[5]=y3+x8}\\
 \indent\texttt{[6]=x5+x6} \\


\end{document}